\documentclass[12pt,reqno]{amsart}

\usepackage{amssymb,latexsym}
\usepackage{enumerate}
\usepackage{pstricks-add}
\usepackage[all,cmtip]{xy}

\newtheorem{lemma}{Lemma}
\newtheorem{definition}[lemma]{Definition}
\newtheorem{proposition}[lemma]{Proposition}
\newtheorem{theorem}[lemma]{Theorem}
\newtheorem{remark}[lemma]{Remark}

\newtheorem{corollary}[lemma]{Corollary}
\newtheorem{example}[lemma]{Example}

\begin{document}

\title{ Flat Affine or Projective Geometries on Lie Groups }


 \author{Medina A.*, Saldarriaga O.**, and Giraldo H.**}


\subjclass[2010]{Primary: 57S20, 53C30; Secondary: 17D25, 53C10 \\ Partially Supported by CODI. Estrategia de Sostenibilidad 2014-2015.}
\date{\today}

\begin{abstract}
This paper deals essentially with affine or projective transformations of Lie groups endowed with a flat left invariant affine or projective structure. These groups are called flat affine or flat projective Lie groups. We give necessary and sufficient conditions for the existence of flat left invariant projective structures on Lie groups. We also determine Lie groups admitting flat bi-invariant affine or projective structures. These groups could play an  essential role in the study of homogeneous spaces $M=G/H$  having a flat affine or flat projective structures invariant under the natural action of $G$ on $M$. A. Medina asked several years ago if the group of affine transformations of a flat affine Lie group  is a flat  projective Lie group. In this work we  provide a partial positive answer to this question.
\end{abstract}
\maketitle

Keywords: Flat affine Lie groups, Flat projective Lie groups, Affine transformations, Projective transformations, Projective \'etale representations.

\vskip5pt
\noindent
 * Universit\'e  Montpellier,  Institute A. Grothendieck, UMR 5149 du CNRS, France and Universidad de Antioquia, Colombia

e-mail: alberto.medina@univ-montp2.fr 
\vskip5pt
\noindent
 ** Instituto de Matem\'aticas, Universidad de Antioquia, Colombia

 e-mails: omar.saldarriaga@udea.edu.co, hernan.giraldo@udea.edu.co

\section{Introduction}

One of the aims of this paper is to give a positive answer  to a question raised by the first author several years ago. More precisely we prove, in several cases, that the group $Aff(G,\nabla^+)$ of affine transformations of a Lie group $G$  endowed with a flat left invariant affine structure  $\nabla^+$, admits a flat left invariant  projective (some times affine)  structure.  

Recall that a locally  affine manifold is a  smooth manifold $M$  endowed with a flat and torsion free linear  connection $\nabla$. This means that the corresponding affine connection $\nabla$ is flat. In this case the pair $(M,\nabla)$ is called a flat affine manifold. We will suppose manifolds to be real connected unless otherwise stated. 

The set  of diffeomorphisms $Aff(M,\nabla)$ of $M$ preserving $\nabla$ endowed with the open compact topology and composition is a Lie group (see \cite{K} page 229). That $F\in Aff(M,\nabla)$ means that  $F$ verifies the system of partial differential equations $F_*(\nabla_XY)=\nabla_{F_*X}F_*Y$,   where $F_*$ is the differential of $F$ and $X$ and $Y$ are smooth vector fields on $M$. In particular $F$ preserves geodesics, but in general the group of geodesic preserving diffeomorphisms of $M$ is larger than $Aff(M,\nabla)$.

The aim of flat affine geometry is the study of flat affine manifolds. The local model of  real flat affine geometry is the n-dimensional real affine space $\mathbb{A}^n$ endowed with the usual affine structure $\nabla^0$. This one is given in standard notation by
\[ \nabla_X^0Y=\sum_{j=1}^nX(g_j)\partial_j,\qquad\text{for}\qquad Y=\sum_{j=1}^ng_j\partial_j\]
with $X$ and $Y$ smooth vector fields in $\mathbb{A}^n$.

From now on we identify $\mathbb{A}^n$ with $\mathbb{R}^n$.  We will often see $\mathbb{R}^n$ as the affine subspace $\{(x,1)\mid x\in\mathbb{R}^n\}$ of $\mathbb{R}^{n+1}$. Hence the classical affine group $Aff(\mathbb{R}^n)=\mathbb{R}^n\rtimes_{{Id}_{GL(\mathbb{R}^n)}} GL({\mathbb{R}}^n)$ of affine transformations of $(\mathbb{R}^n,\nabla^0)$ will be identified with a closed subgroup of $GL(\mathbb{R}^n\oplus\mathbb{R})$.

Notice also that every invariant pseudo metric in $\mathbb{R}^n$ determines the same geodesics as $\nabla^0$.

\begin{definition} A flat affine (respectively a flat projective) Lie group is a Lie group endowed with a flat left invariant affine structure (respectively flat left invariant projective structure). The first one will be abreviated as FLIAS. The corresponding infinitesimal object is called an affine Lie algebra (respectively a projective Lie algebra). \end{definition} 

For the local model of flat projective geometry see Section \ref{S:preliminaries}.

An important and difficult open problem is to determine whether a manifold (respectively a Lie group) admits a flat (repectively left invariant) affine or projective structure. Obviously a manifold $H\backslash G$, where $H$ is a discrete co-compact  subgroup of a flat  affine Lie group $G$, inherits a flat affine structure. Let $G$ be a Lie group of Lie algebra $\mathfrak{g}:=T_\epsilon(G)$ with $\epsilon$  the unit of $G$. If $x\in T_\epsilon(G)$, the left invariant vector field determined by $X$ will be denoted by $x^+$. Having a real bilinear product on $\mathfrak{g}$ is equivalent to having a left invariant linear connection $\nabla$ on $G$. The connection $\nabla$ is given by defining $\nabla_XY:=(X_\epsilon\cdot Y_\epsilon)^+$, where $X$ and $Y$ are left invariant vector fields, and extending it so that $\nabla_X(gY)=X(g)Y+g\nabla_XY$, where $g$ is any smooth function on $G$. For instance, the (0)-Cartan connection on a Lie group $G$ is defined by the bilinear product $x\cdot y=\frac{1}{2}[x,y]$ on $\mathfrak{g}$. Having a FLIAS on $G$ is equivalent to having a left symmetric product on $\mathfrak{g}$ compatible with the bracket, i.e., a bilinear product $\cdot$ verifying
\begin{align} \label{E:zerocurvature1}
&(x, y, z)=(y,x,z), \\ \label{E:freeoftorsion1}
&x\cdot y-y\cdot x=[x,y]
\end{align}
where $(x,y,z)$ is the associator of $x,y$ and $z$ in $\mathfrak{g}$ and $[x,y]$ is the bracket given in $\mathfrak{g}$. The pair $(\mathfrak{g},\cdot)$ is called a left symmetric algebra (LSA). This is also equivalent to having an affine \'etale representation of $G$, i.e., a representation with an open orbit and discrete isotropy (see \cite{Kz} and \cite{M}). The group $Aff(G,\nabla^+)$ has been studied in \cite{BM1} in the case where $\nabla^+$ is bi-invariant. Sometimes the group $Aff(G,\nabla^+)$ is also a flat affine Lie group. In particular $Aff(\mathbb{R}^n,\nabla^0)$ is a flat affine group since its co-adjoint representation is \'etale (\cite{BY}, \cite{BM1}, \cite{R}). This fact and other observations led the first author of this paper to ask whether the group $Aff(G,\nabla^+)$ admits a FLIAS or by default a flat left invariant projective structure. 

\noindent
One of the main tools used in our work is the following  well known result (see \cite{FGH} and \cite{Kz}).

Denote by $p:\widehat{M}\longrightarrow M$  the universal covering map of a real $n$-dimensional flat affine manifold $(M,\nabla)$. The pullback $\widehat{\nabla}$  of $\nabla$ by $p$ is a flat affine structure on $\widehat{M}$ and $p$ is an affine map. Moreover, the group $\pi_1(M)$ of deck transformations acts on $\widehat{M}$ by affine transformations. 
\begin{theorem} \label{T:developant} \textbf{(Development Theorem)} There exists an affine immersion $D:\widehat{M}\longrightarrow \mathbb{R}^n$, called the developing map of $(M,\nabla)$, and a  group homomorphism  $A:Aff(\widehat{M},\widehat{\nabla}) \longrightarrow Aff(\mathbb{R}^n,\nabla)$ so that the following diagram commutes
\[ \xymatrix{ \widehat{M} \ar[d]_{F} \ar[r]^{D} &\mathbb{R}^n\ar[d]^{A(F)}\\
\widehat{M} \ar[r]^{D} &\mathbb{R}^n.} \]
\end{theorem}
In particular for every $\gamma\in\pi_1(M)$ we have $D\circ \gamma= H(\gamma)\circ D$ where $H(\gamma):=A(\gamma)$ and $H$ being a group homomorphism. This last homomorphism is called the holonomy representation of $(M,\nabla)$. 

This result seems to be due to Ehresman (see \cite{E}).

Let us consider a  homogeneous space $(S,T)$, i.e., a manifold $S$ and   a Lie group $T$  acting transitively  on $S$. A manifold $M$ is provided of an $(S,T)$-geometry if it admits a smooth subatlas $\{\left(U_i,\varphi_i\right)\}$ so that $\varphi_i(U_i)\subseteq S$ with $\varphi_i^{-1}\circ\varphi_j$ the restriction of an element of $T$, whenever $U_i\cap U_j\ne\emptyset.$ 

\begin{remark} \label{R:inducedgeometrybyembedding} Suppose given an embedding  $(\theta,\rho):(S,T)\longrightarrow (S',T')$ of homogeneous spaces, that is, embeddings $\theta:S\longrightarrow S'$ and $\rho:T\longrightarrow T'$  with $\rho$ a group  homomorphism    so that the following  diagram  commutes
\[ \xymatrix{
S \ar[d]_{t} \ar[r]^{\theta} &S'\ar[d]^{\rho(t)}\\
S \ar[r]^{\theta} &S'} \]
for any $t\in T$. Then it is is clear that an $(S,T)-$geometry determines an $(S',T')-$geometry. In particular an affine geometry determines a  projective geometry. Consequently a flat affine Lie group is a flat projective Lie group. \end{remark}

\begin{remark} \label{R:locallyisom}
Suppose that $G$ and $G'$ are locally isomorphic Lie groups, i.e., $G$ and $G'$ have isomorphic Lie algebras. If $G$ has a flat left invariant (or bi-invariant)  affine or projective  structure then so does $G'$. In the affine case, if the  structure is left invariant, the Lie bracket is the commutator of a left symmetric product. If the structure is bi-invariant, the bracket is the commutator of an associative product (see \cite{M} and \cite{BM1}).\end{remark} 

The main two problems of flat affine (or projective) geometry are to find conditions for the existence of a flat affine (or projective) structure on a manifold, and once such a structure exists on a particular manifold $M$, to determine all the flat affine (or projective) structures on $M$. Benz\'ecri  classified the flat affine structures on closed surfaces (see \cite{B}).
In general it seems hopeless to try to classify general geometric structures on noncompact manifolds (see \cite{G}). There are not known sufficient and necessary conditions for the existence of a flat affine structure on a manifold (see \cite{Sm}, \cite{Sm1}, \cite{KS} and \cite{ST}).

\begin{definition} A Lie group endowed with a left invariant symplectic (respectively K\"ahler) structure is called a symplectic (respectively K\"ahler) Lie group.
\end{definition} 
A symplectic Lie group with a symplectic form $\omega^+$ is always a flat affine Lie group. For instance, Equation \eqref{Eq:leftsymmetricproductbyasymplecticform} determines a flat affine structure on a symplectic Lie group. In particular, a K\"ahler Lie group is a flat affine Lie group.

Symplectic or K\"ahler Lie groups have been studied in \cite{DM}, \cite{Hj}, \cite{LM},  \cite{MR}, \cite{M1}, \cite{DM1} and \cite{L}.  

 The paper is organized as follows. In Section \ref{S:preliminaries} we recall some basic notions about affine and projective geometry to facilitate the reading of this paper. We also describe a method to construct pseudo-K\"ahler Lie groups from pseudo-Hessian Lie groups. Section \ref{S:6}  answers possitively Medina's question in some particular cases. Section \ref{S:LIASLG} shows that there are infinitely many non-isomorphic FLIAS on $G=Aff(\mathbb{R})$. In Section \ref{S:lastsection} we show that, for every FLIAS $\nabla^+$ on $G=Aff(\mathbb{R})$,  the group of affine tranformations $Aff(G,\nabla^+)$ is a  symplectic Lie group and therefore a flat affine Lie group. In Section \ref{S:specialleftsymmetricproducts} we explicit FLIAS on $G$ of special interest in geometry, such as flat Hessian structures, flat Lorentzian structures, affine symplectic structures,    and Kh$\ddot{\text{a}}$lerian structures. The Appendix exhibits the affine \'etale representations and geodesics corresponding to each FLIAS on $G=Aff(\mathbb{R})$.



The reader can refer to \cite{M} for elements on flat affine Lie groups and to \cite{LM}, \cite{MR}, \cite{DM} and \cite{L} for elements on symplectic or  K\"ahler Lie groups. 

\section{Preliminary   remarks } \label{S:preliminaries}

Let $M$ be an $n$-dimensional real manifold. Two torsion free affine connections $\nabla$ and $\nabla'$ on $M$ are projectively equivalent if there exists a smooth 1-form $\phi$ on $M$ verifying $\nabla_XY-\nabla'_XY=\phi(X)Y+\phi(Y)X$ for any smooth vector fields $X$ and $Y$ on $M$. This means that both connections have the same geodesics up to parametrization. The equivalence class $[\nabla]$ of $\nabla$ under this relation is called a projective structure on $M$. A projective isomorphism is a diffeomorphism $F:M\longrightarrow M'$ so that $F^*([\nabla'])=[\nabla]$. A projective structure $[\nabla]$ on $M$ is called projectively flat if for each point $p\in M$ there exists a neighborhood $U$ of $p$ and a diffeomorphism $f:U\longrightarrow f(U)\subseteq \mathbb{R}^n$, so that $f^*([\nabla^0])=[\nabla]$. This means that the Cartan normal connection is flat (see \cite{C} and \cite{K1}). In particular the Weyl tensor vanishes (see \cite{W}, \cite{Ei} and \cite{Th}).

For what follows let us recall also the following basic facts. Let $A(M)$ (respectively $L(M)$) be the bundle of affine frames (respectively linear frames) on an $n$-dimensional manifold $M$. The split exact sequence of Lie groups 
\[ 0\longrightarrow  \mathbb{R}^n \longrightarrow  Aff(\mathbb{R}^n)\begin{matrix} \overset{\gamma}\longleftarrow\\\underset{\beta}\longrightarrow\end{matrix}  GL(\mathbb{R}^n)\longrightarrow  1  \] 
determines morphisms of principal bundles $\hat{\beta}:A(M)\longrightarrow L(M)$ and  $\hat{\gamma}:L(M)\longrightarrow A(M)$ such that $\hat{\beta}\circ\hat{\gamma} =id$. So, if $\widehat{\omega}$ is a connection form on $A(M)$ (i.e., a generalized affine connection) we have that \[\hat{\gamma}^*\widehat{\omega}=\omega+\varphi,\] where $\omega$  is a $gl(n,\mathbb{R})$-valued 1-form and $\varphi$ is an $\mathbb{R}^n$-valued 1-form on $L(M).$ In the case where $\varphi$ is the canonical form on $L(M)$, the $1$-form $\widehat{\omega}$ is called an affine connection on $M$. Hence, if $\Theta$ and $\Omega$ are respectively the torsion and the curvature forms of a linear connection on $M$, we will have that $\hat{\gamma}^*\widehat\Omega=\Theta+\Omega$, where $\widehat{\Omega}$ is the curvature form of an affine connection $\omega.$ In summary, a flat and torsion free linear connection can be viewed as a flat affine connection on $M$ (see \cite{K}). 

E. Cartan studied (and solved) the problem of determining when the (0)-Cartan connection is projectively flat (see \cite{C} and \cite{N}). 


\begin{remark}  It is well known that the Levi-Civita connection corresponding to a pseudo-Riemannian metric is projectively flat if and only if its scalar curvature is constant (see \cite{Ei}). 

In particular the  $n$-dimensional real projective space $P_n$ has a natural flat projective structure.  The model space of the real flat projective geometry is the pair $(P_n,Aut(P_n)).$ In what follows we will identify $P_n$ with the homogeneous space $Aut(P_n)/Is$, where $Is$ denotes the isotropy in any point of $P_n$ for the natural action.
\end{remark} 

Notice that $Aut(P_n)=SL(n+1,\mathbb{R})/\text{center}$ does not admit flat left invariant affine connection. E. Vinberg in \cite{V} developed a theory for projectively homogeneous bounded domains in $\mathbb{R}^n$ (see also \cite{N}). The study of projective geometry can be done via Cartan connections (see \cite{C}, \cite{K1} and \cite{Ag}) or using the theory of jets due to Ehresman (\cite{E}, \cite{K1} and \cite{T}).

Every left invariant (respectively bi-invariant) pseudo-Riemannian metric on a Lie group determines a left invariant (respectively bi-invariant) projective structure (non necessarily flat).

Next we give a sufficient and necessary condition for the existence of a flat left invariant projective structure on a Lie group.

\begin{theorem} Let $G$ be a real $n$-dimensional connected Lie group, $P_n$ the n-dimensional projective space over the field $\mathbb{R}$  and $Aut(P_n)$ the group of projective transformations of $P_n$. The group $G$ admits a flat left invariant projective structure if and only if there exists a Lie group homomorphism $\rho:G\longrightarrow Aut(P_n)$ having an open orbit with discrete isotropy, that is, $\rho$ is a projective \'etale representation of $G$.
\end{theorem}
\begin{proof} Let us suppose that $G$ is endowed with a flat left invariant projective structure. Let $F^2(G)$ be the bundle of 2-frames on $G$. The manifold $F^2(G)$ is a fiber bundle over $G$ with structure group $J_\epsilon^2(g)$, the group of 2-jets with $g$ a local diffeomorphism of $G$ fixing $\epsilon.$ The projective structure $P$ is a subbundle of the bundle $F^2(G)$ with structure group isomorphic to $Is_0$, where $Is_0$ is the subgroup of automorphisms of $P_n$ fixing a point $0$, called the origin.  

The transformation $L_\sigma$ of $G$ induces, in a natural manner, an automorphism $(L_\sigma)_*:F^2(G)\rightarrow F^2(G)$. That the projective structure is left invariant, means that the restriction of $(L_\sigma)_*$ to $P$ is an automorphism of $P$, this means that the  following diagram commutes
\[ \xymatrix{
P \ar[d]_{\Pi} \ar[r]^{(L_\sigma)_*} &P\ar[d]^{\Pi}\\
G \ar[r]^{L_\sigma} &G.} \]
The group $Aut(P_n)$ can be considered as a fiber bundle over $P_n$ with structure group $Is_0$. That the structure $P$ is flat means that the fiber bundles $P$ and $Aut(P_n)$ are isomorphic, i.e., there is a commutative diagram 
\[ \xymatrix{
P \ar[d]_{\Pi} \ar[r]^{\psi_*} &Aut(P_n)\ar[d]^{\Pi}\\
G \ar[r]^{\psi} &P_n.} \]

From the isomorphism of fiber bundles, we get that to any diffeomorphism $L_\sigma$ of $G$, there corresponds $\rho(\sigma)\in Aut(P_n)$. Since $L_\sigma\circ L_\tau=L_{\sigma\tau}$, it follows that $\rho$ is a Lie group homomorphism. 

To show that $\rho$ is an \'etale representation, let us choose $\{(U_j,\phi_j:U_j\longrightarrow O_j\subseteq P_n)_{j\in J}\}$  a projective atlas, that is, a smooth atlas so that $\phi_j\circ \phi_i^{-1}\in Aut(P_n)$, whenever $U_j\cap U_i\ne\emptyset$. Notice that the atlas can be chosen so that $\Pi^{-1}(U_j)$ is diffeomorphic to $U_j\times Aut(P_n)$. 

Let $\gamma$ be the natural action of $G$ on $P_n$ associated to $\rho$, that is,
\[ \begin{array}{ccclc} \gamma:&G\times P_n&\longrightarrow &P_n\\ &(\sigma,[v])&\mapsto&\rho(\sigma)([v]),\end{array}\]
where $[v]\in\mathbb{R}^{n+1}\setminus\{0\}/\mathbb{R}^*$. By  choosing $(U,\phi:U\longrightarrow O)$ a local system of projective coordinates and $[v]\in O$, then there exists a unique $\tau$ so that $\phi(\tau)=[v]$. It is easy to see that $[v]$ is of open orbit (and discrete isotropy).

Conversely,  let $\rho:G\longrightarrow Aut(P_n)$ be a projective \'etale representation of $G$. If $[v]\in P_n$ is of open orbit and discrete isotropy, then $Orb([v])$ inherits a flat projective structure from those of $P_n$. Moreover, the orbital map $\pi:G\longrightarrow Orb([v])$ defined by  $\pi(\sigma)=\rho(\sigma)([v])$ is a covering map. It follows that the pullback of the flat projective structure on $Orb([v])$ lifts to a flat projective structure on $G$. The left invariance of this structure follows from the equivariance of the map $\pi$  by the actions of $G$ on $G$ by left multiplications and  of $\rho(G)$ on $Orb([v])$ and from the fact the following diagram  is commutative 
\begin{equation} \label{Eq:projectiverep} \xymatrix{
G \ar[d]_{\pi} \ar[r]^{L_\tau} &G\ar[d]^{\pi}\\
Orb([v]) \ar[r]^{\rho(\tau)} &Orb([v]).} \end{equation}
\end{proof}

From the arguments used  in the previous proof we can deduce the following.

\begin{corollary} Let $G$ be a real Lie group of dimension $n$ with $\mathfrak{g}=Lie(G)$. The following statements are equivalent.
\begin{enumerate}
\item $G$ admits a flat left invariant projective structure.
\item there exists a representation $\rho:G\longrightarrow Aut(P_n)$ and a point $[p]\in P_n$ whose orbit is open and with discrete isotropy, i.e., $G$ admits a  projective \'etale representation.
\item there exists a linear representation $\theta:\mathfrak{g}\longrightarrow sl(n+1,\mathbb{R})$ and a point $w\in\mathbb{R}^{n+1}$ so that $\theta(\mathfrak{g})(w)+\mathbb{R}w=\mathbb{R}^{n+1}$ (see \cite{El}).
\end{enumerate}
\end{corollary}
\begin{proof} Assertions 1. and 2. are equivalent from the previous theorem. Now we show 2. implies 3. Let $p:\mathbb{R}^{n+1}\setminus\{0\}\longrightarrow P_n$  be the canonical projection. Since $Orb([v])$ is open in $P_n$, then $p^{-1}(Orb([v]))$ is  open in $\mathbb{R}^{n+1}$ saturated by the relation $\sim$ determined by $u\sim v$ if and only if $v=\lambda u$, with $\lambda\ne0$. From this and the fact that  $\rho_{*,\epsilon}:\mathfrak{g}\longrightarrow sl(n+1,\mathbb{R})$ is a Lie algebra homomorphism, it follows  the  existence of $w\in \mathbb{R}^{n+1}$ verifying  $\rho_{*,\epsilon}(\mathfrak{g})(w)+\mathbb{R}w=\mathbb{R}^{n+1}$. 

The proof of 3. implies 2. follows by using the exponential map. 
\end{proof}

\begin{remark} It is well known that a semisimple Lie group $G$ of finite dimension $m$ does not admit a FLIAS (see \cite{H}). This follows from observing that such a structure on $G$ has to be radiant and this implies $m=0$. In particular $SL(n+1,\mathbb{R})$ does not admit a FLIAS, nevertheless it does admit a flat left invariant projective structure. \end{remark}

\begin{example} Consider the group homomorphism $\rho:SL(2,\mathbb{R})\longrightarrow Aut(P_3)$ defined by $\rho(u)=\left[\begin{matrix} u&0\\0&I_2\end{matrix}\right]$. Viewing $\rho$ as a homomorphism $\rho:SL(2,\mathbb{R})\longrightarrow SL(3,\mathbb{R})\subseteq GL(4,\mathbb{R})$, it is easy to verify that $\rho_{*,\epsilon}:sl(2,\mathbb{R})\longrightarrow sl(3,\mathbb{R})\subseteq gl(4,\mathbb{R})$ is given by $\rho_{*,\epsilon}(u')= \left[\begin{matrix} u'&0\\0&E\end{matrix}\right]$   where $E=\left[\begin{matrix}0&0\\0&1\end{matrix}\right]$ and that  $\rho_{*,\epsilon}(sl(2,\mathbb{R}))+\mathbb{R}e_3=\mathbb{R}^4$, where $e_3=(0,0,1,0)$. Therefore $\rho$ is a projective \'etale representation of $SL(2,\mathbb{R})$. 

Consequently $SL(n,\mathbb{R})$, with $n\geq3$, admits a flat left invariant projective structure. 
\end{example}

We now recall the definition of pseudo-K\"ahler Lie group, then we show that the symplectic  cotangent Lie group of a Lie group equipped with  a pseudo-Hessian structure is a pseudo-K\"ahler Lie group.

\begin{definition} \label{D:pseudokhaler} Let $(G,\Omega^+)$ be a symplectic Lie group. If $G$ is provided of a left invariant complex structure $j^+$ so that 
\[ g^+(a,b):=\omega(a,j^+(b)) \] 
defines a left invariant pseudo-Riemannian metric on $G$, then  $(G,g^+,j^+)$ is called a  pseudo-K\"ahler Lie group. 


\end{definition}
For the study of symplectic or pseudo-K\"ahler Lie groups   see \cite{MR1}, \cite{DM}, \cite{DM1}  and \cite{L}.

\begin{theorem} \label{P:cotangentaskalerliegroup} Given a  flat affine pseudo-Hessian Lie group   $(G,\nabla^+,g^+)$  the cotangent bundle $T^*(G)=\mathfrak{g}^*\rtimes_{\rho}G$, endowed with the product $(\alpha,\sigma)(\beta,\tau)=(\alpha+\rho(\sigma)(\beta),\sigma\tau)$ where $\rho_{*,\epsilon}(a)=-{}^tL_a,$ for $a\in\mathfrak{g}$, is a pseudo-K\"ahler Lie group. 
\end{theorem}
\begin{proof} Denote by $t^*\mathfrak{g}$ the Lie algebra of $T^*G$. The dual space $\mathfrak{g}^*$ of the left symmetric algebra $\mathfrak{g}$ is a $\mathfrak{g}$-bimodule by the actions $\alpha\cdot b=\alpha\circ R_b$ and $b\cdot\alpha=ad^*(b)(\alpha)$. Then the vector space $t^*\mathfrak{g}=\mathfrak{g}^*\times \mathfrak{g}$ can be viewed as a left symmetric algebra, extensi\'on of the left symmetric algebra $\mathfrak{g}$ by the $\mathfrak{g}$-bimodule $\mathfrak{g}^*$ as described in \cite{MR}. The pair  $(t^*\mathfrak{g},\omega)$, with $\omega((\alpha,\sigma),(\beta,\tau))=\alpha(\tau)-\beta(\sigma)$, is a symplectic Lie algebra. The restriction on $\mathfrak{g}$ of the product determined by \eqref{Eq:leftsymmetricproductbyasymplecticform} agrees with the given product of $\mathfrak{g}$. On the other hand, the form defined by 
\[ \begin{array}{lccc} j(x)=-g(x,\cdot),&\text{  for  }&x\in\mathfrak{g}\\ 
j(\alpha)=x &\text{ if }&\alpha=g(x,\cdot) \end{array}\]
is an integrable complex structure. Finally, $\langle x,y\rangle:=\omega(x,j(y))$ is a non-degenerate bilinear form on $t^*\mathfrak{g}$. 
\end{proof}

\section{Projective or affine geometries on  transformation groups } \label{S:6}

Results on this section give a positive answer to  Medina's question in some particular cases.

\begin{theorem} \label{T:structureonglnliftedtoaffine} The $\mathbb{R}-$bilinear product  on $aff(\mathbb{R}^n)=\mathbb{R}^n\rtimes_{id}gl(\mathbb{R}^n)$ given by
\begin{equation} \label{Eq:productoasociativoenaff} (a,f)\cdot (b,g)=(f(b),f\circ g) \end{equation}
is associative and compatible with the Lie bracket of $aff(\mathbb{R}^n)$. Hence it defines a flat bi-invariant structure $\widetilde{\nabla}$ on the group $Aff(\mathbb{R}^n,\nabla^0)$. Moreover, if $\circ$ denotes the linear connection determined by composition of linear endomorphisms, then the canonical sequence 
\[ 0\longrightarrow (\mathbb{R}^n,\nabla^0)\longrightarrow (Aff(\mathbb{R}^n,\nabla^0),\widetilde{\nabla})\longrightarrow GL(\mathbb{R}^n,\circ)\longrightarrow 1\]
is a split exact  sequence of flat bi-invariant affine Lie groups. 
\end{theorem}
\begin{proof} A direct calculation shows that the product \eqref{Eq:productoasociativoenaff} is an associative product on $aff(\mathbb{R}^n)=\mathbb{R}^n\rtimes_{id}gl(\mathbb{R}^n)$ compatible with the Lie bracket on $aff(\mathbb{R}^n)=\mathbb{R}^n\oplus_{Id}gl(n,\mathbb{R})$ (see \eqref{E:zerocurvature1} and \eqref{E:freeoftorsion1} above). 

When $u,v=0$ we get $\nabla^0$. Moreover, the subspace $\{(x,0)\mid x\in\mathbb{R}^n\}$ is a bilateral ideal of the associative algebra $aff(\mathbb{R})$ with the product \eqref{Eq:productoasociativoenaff}. Moreover the map $f\mapsto (0,f)$ is a section of the exact sequence.
\end{proof}


\begin{remark}  Analogous results are obtained for the real Lie group $GL(n,\mathbb{K})$ where  $\mathbb{K}=\mathbb{C}$ or  $\mathbb{H}$, the  field (non-commutative) of quaternions. Also  notice that similar results to the previous theorem hold even if the connection on $GL(n,\mathbb{R})$ is assumed only left invariant.
\end{remark}

\begin{proposition} \label{P:completaimplicabiinvariante} If $(G,\nabla^+)$ is a complete flat affine Lie group  then the group $Aff(G,\nabla^+)$ admits a non-complete  flat bi-invariant affine structure and a left invariant symplectic structure. 
\end{proposition}
\begin{proof} Using the covering map $p:\widehat{G}\longrightarrow G$ and the $p$-pullback of $\nabla^+$ we can suposse that $G$ is simply connected. Denote by $\theta:\mathfrak{g}\longrightarrow aff(\mathfrak{g})$ the affine representation of  $\mathfrak{g}=Lie(G)$ defined by $\theta(x)=(x,L_x)$ and by $\rho:G\longrightarrow Aff(\mathfrak{g})$  the corresponding \'etale affine representation. If $\rho$ is given by $\rho(\sigma)=(Q(\sigma),F_\sigma)$ then $Q:G\longrightarrow\mathfrak{g}$ is the developing map of $(G,\nabla^+)$. Hence there exists a homomorphism of Lie groups $A:Aff(G,\nabla^+)\longrightarrow Aff(\mathfrak{g},\nabla^0)$  such that the following diagram commutes
\[ \xymatrix{ G \ar[d]_{F} \ar[r]^{Q} &\mathfrak{g}\ar[d]^{A(F)}\\
G\ar[r]^{Q} &\mathfrak{g}} \]
Moreover, since $\nabla^+$ is complete, the map $Q$ is an affine diffeomorphism. We need to prove that $A$ is an isomorphism.

Let us suppose that $F\in\ker(A)$, i.e., $A(F)=Id_{\mathfrak{g}}$. As the diagram commutes  we get that $F=Id_G$. On the other hand, as $Aff(G,\nabla^+)$ is complete, $\dim Aff(G,\nabla^+)=\dim Aff(\mathfrak{g},\nabla^0)$. Consequently $A$ is an isomorphism of Lie groups. 

Since $Aff(\mathfrak{g},\nabla^0)$ has a flat bi-invariant affine structure (Theorem \ref{T:structureonglnliftedtoaffine}) and a left invariant symplectic structure (the co-adjoint representation of $Aff(\mathfrak{g},\nabla^0)$ is \'etale, see \cite{R} and \cite{BM}),  by the pullback by $A$, we get the result.  
\end{proof}

Let $G$ be a Lie group endowed with a flat left invariant pseudo-Riemannian metric $g^+$ and let $\nabla^+(g)$ denote the Levi-Civita connection associated to $g^+$. Then we have the following.

 \begin{corollary}  If $G$ is unimodular then  $Aff(G,\nabla^+(g))$ admits a flat bi-invariant affine structure and a left invariant symplectic structure. 
\end{corollary}
 \begin{proof} A flat left invariant pseudo-Riemannian metric on a Lie group $G$ is complete if and only if $G$ is unimodular (see \cite{AM}). Then  from Proposition \ref{P:completaimplicabiinvariante}, $Aff(G,\nabla^+(g))$ admits a flat bi-invariant structure. The existence of a left invariant symplectic structure on $Aff(G,\nabla^+(g))$ follows from \cite{BM}, \cite{R} and the fact that  the structure on $G$ is complete. 
\end{proof}

Consider the $Int(\mathfrak{g})$-structure $P$ on a Lie group $G$ with Lie algebra $\mathfrak{g}$ obtained by the natural right action of $Int(\mathfrak{g})$ on the left invariant parallelism $(e_1^+,\dots,e_n^+)$ of $G$ (see \cite{M}). It is clear that the product $\nabla_xy=ad_{f(x)}(y)$, where $f\in gl(\mathfrak{g})$, determines a left invariant connection $\nabla_f^+$ on $G$ adapted to $P$.  We have the following proposition.

\begin{corollary} If $\nabla_f^+$ is flat and torsion free, then the group $Aff(G,\nabla_f^+)$ is a flat affine Lie group.
\end{corollary}
\begin{proof} It is easy to verify that right multiplications $R_a$, $a\in\mathfrak{g}$, are nilpotent, hence  $\nabla_f^+$ is geodesically complete (see \cite{M}).
\end{proof}

\begin{theorem} \label{T:biinvariantstructureonlineargroup}  Let $G=GL(n,\mathbb{K})$ be the real Lie group, with $\mathbb{K}=\mathbb{R}$ or $\mathbb{H}$,  and $\nabla^+$ the flat bi-invariant affine structure determined by composition of linear endomorphisms of $\mathbb{K}^n$. Then the group $Aff(G,\nabla^+)$ admits a flat bi-invariant projective structure.
\end{theorem} 
\begin{proof} For every $\sigma\in G$, let $L_\sigma$ and $R_\sigma$ denote the maps left and right multiplication by $\sigma$, respectively, and $I_\sigma=L_\sigma\circ R_{\sigma^{-1}}$. Since $\nabla^+$ is bi-invariant, we have that $L_\sigma,R_\sigma,I_\sigma \in Aff(G,\nabla^+)$. 

Let $Z$ be the center of $G$, $M=G\times G/Z$  and $\psi:M\longrightarrow Diff(G)$ the map defined by $\psi(\tau,\overline{\sigma})=L_\tau\circ I_\sigma$. An easy calculation shows that $\psi$ is an injective Lie group homomorphism, and therefore it is an immersion. Let us consider the group structure on $\psi(M)$ determined by composition and define the product on $M$ by 
\[ (\tau_1,\overline{\sigma_1})\cdot (\tau_2,\overline{\sigma_2})=\psi^{-1}\left(\psi(\tau_1,\overline{\sigma_1})\circ \psi(\tau_2,\overline{\sigma_2})\right).\] 
It can be checked that, under this product, $M$ is a Lie group, $G$ is a normal subgroup of $M$ and $G/Z$ is a subgroup. 

Let $\theta:G/Z\longrightarrow Int(G)$ be the map defined by $\theta(\sigma)=I_\sigma$. It is easy to check that $$\psi(\tau_1,\overline{\sigma_1})\circ \psi(\tau_2,\overline{\sigma_2})=\psi(\tau_1I_{\sigma_1}(\tau_2),\overline{\sigma_1\sigma_2})=\psi(\tau_1\theta(\sigma_1)(\tau_2),\overline{\sigma_1\sigma_2})$$ therefore $\psi^{-1}(M)=G\rtimes_\theta G/Z$. The Lie algebra of $G\rtimes_\theta G/Z$ is $gl(n,\mathbb{R})\oplus_{ad} sl(n,\mathbb{R})$, the semidirect product of the Lie algebras $gl(n,\mathbb{R})$ and $sl(n,\mathbb{R})$ where $ad:sl(n,\mathbb{R})\longrightarrow End(gl(n,\mathbb{R}))$ is the restriction to $sl(n,\mathbb{R})$ of the adjoint map. Hence $\dim M=2n^2-1$. Since $\dim Aff(G,\nabla^+)=2n^2-1$ (see \cite{BM1})  we get that $M\cong\psi(M)$ and $Aff(G,\nabla^+)$ are locally isomorphic. Hence, by Remark \ref{R:locallyisom}, $M$ admits a flat bi-invariant structure. 

It can be verified that the product $(t_1,s_1)\cdot (t_2,s_2)=(t_1\circ t_2+t_1\circ s_2+s_1\circ t_2,s_1\circ t_1)$ in $gl(n,\mathbb{R})\oplus gl(n,\mathbb{R})$ is associative and compatible with bracket in $gl(n,\mathbb{R})\oplus_{ad} gl(n,\mathbb{R})$. The associative algebra $(gl(n,\mathbb{R})\oplus gl(n,\mathbb{R}),\cdot)$ is unitary with unit $(0,I)$, where $I$ is the identity endomorphism of $gl(n,\mathbb{R})$. One can verify  that 
\[ \{ L_{(s,t)}\mid tr(L_{(s,t)})=0\}=sl(n,\mathbb{R})\oplus sl(n,\mathbb{R})\oplus \mathbb{R}(2I,-I)\cong gl(n,\mathbb{R})\oplus_{ad} sl(n,\mathbb{R}),\]
where $L_{(s,t)}$ is left multiplication by $(s,t)$. Hence by Theorem 2.2 in \cite{Sh2} (see also the converse to Theorem 2 in \cite{N} stated after the proof of that theorem) the group $Aff(G,\nabla^+)$ admits a flat bi-invariant projective structure.

Mutatis mutandis we prove for $\mathbb{H}$.
\end{proof} 

Recall  that an associative algebra $A$ over a field $\mathbb{F}$ is called central if its center $C$ is given by $C:=\mathbb{F}1=\{d1\mid d\in\mathbb{F}\} $  (see \cite{J}). If $A$ is a finite dimensional algebra over $\mathbb{F}$ then $A\cong M_r(\Delta)=\{$ set of $r\times r$ matrices with entries in $\Delta\}$, where $\Delta$ is a finite dimensional division algebra over $\mathbb{F}$. In this case the center $C$ of $A$ is isomorphic to the center of $\Delta$  and hence $C$ is a field. So, $A$ can be regarded as an algebra over $C$. When this is done $A$ becomes a finite dimensional simple central algebra over $C$. The division algebras over $\mathbb{F}=\mathbb{R}$ are the fields of  real numbers $\mathbb{R}$, complex numbers $\mathbb{C}$ or  quaternions $\mathbb{H}$ (Theorem of Hurewitcz). 

In what follows $G=U(A)$ is the real group of units of a finite dimensional real associative algebra $A$. It is clear that $G$ is endowed with a natural flat bi-invariant affine structure $\nabla^+$. The group $Aff(G,\nabla^+)$ has been studied in \cite{BM1}. We have 

\begin{proposition} If $A=M_n(\mathbb{K})$, with $\mathbb{K}=\mathbb{R}$ or $\mathbb{H}$ and $G=U(A)$, then the natural structure $\nabla^+$ on $G$ determines a natural flat left invariant projective structure on $Aff(G,\nabla^+)$. 
\end{proposition}
\begin{proof} Every  quaternion can be represented as a real  4$\times$4 matrix. Hence every element of $A$ can be seen as a real matrix and $A$ is a real associative algebra. It is obvious that $G=U(A)$ admits a flat bi-invariant affine structure determined by the product on $A$. This structure determines a flat bi-invariant projective structure on $G$. In particular, this projective structure is invariant under $Int(G)=GL(n,\mathbb{K})/C=SL(n,\mathbb{K})$. The second assertion can be proved using an argument similar to that given in the proof of Theorem \ref{T:biinvariantstructureonlineargroup}.
\end{proof}

\begin{remark} In fact $SL(n,\mathbb{R})$ (respectively $SL(n,\mathbb{H})$)  admits a flat bi-invariant projective structure. These are the only semisimple real Lie groups admitting such a structure (see \cite{N}).
\end{remark}

\subsection{Remarks on complex projective or affine geometry}

Using an argument similar to that used in the proof of Theorem \ref{T:structureonglnliftedtoaffine} we get the following result.

\begin{proposition}  The  holomorphic flat bi-invariant affine structure on $GL(n,\mathbb{C})$ determined by composition of $\mathbb{C}$-linear endomorphisms can be lifted to a  flat bi-invariant affine holomorphic structure on $T=Aff(\mathbb{C}^n)$  so that the induced structure on $\mathbb{C}^n$ is $\nabla^\circ$. Moreover  the sequence 
\[ 0\longrightarrow \mathbb{C}^n\longrightarrow Aff(\mathbb{C}^n)\longrightarrow GL(n,\mathbb{C})\longrightarrow Id \]
is an exact sequence of complex flat affine  Lie groups and consequently an exact sequence of flat projective Lie groups. \end{proposition}

Let us consider $P_n(\mathbb{C})$ endowed with the Fubini-Study metric. Since the holomorphic sectional curvature is constant, the manifold $P_n(\mathbb{C})$ has a flat projective structure (see \cite{Gb}). The group of projective transformations $GL(n+1,\mathbb{C})/\mathbb{C}^*I_n$ of $P_n(\mathbb{C})$ can be identified to $SL(n+1,\mathbb{C})/center$. This group does not admit flat left invariant affine structures (see for instance \cite{H}). Nevertheless, since $sl(n,\mathbb{C})$ can be realized as the subalgebra of $gl(n+1,\mathbb{C})$ formed by elements $a$ so that $trace(L_a)=0$, where $L_a$ is left composition of $\mathbb{C}-$endomorphisms, it follows that $SL(n+1,\mathbb{C})/center$ admits a flat bi-invariant  projective structure.

\begin{proposition}  The complex Lie group $Aff(G,\nabla)$, with $G=GL(n+1,\mathbb{C})$ and $\nabla$ the connection determined by composition of complex linear endomorphisms, admits a flat bi-invariant projective structure.
\end{proposition}
\begin{proof} The proof is similar to that given in Theorem \ref{T:biinvariantstructureonlineargroup}.
\end{proof}

\section{Flat left invariant affine structures on $Aff(\mathbb{R})$} \label{S:LIASLG} 

In this section we show that there exists infinitely many non-isomorphic flat affine structures on $G=Aff(\mathbb{R})$. In Section \ref{S:lastsection} we prove that each group of affine transformations preserving these structures admits a left invariant symplectic structure, and therefore a flat left invariant affine structure.

The product on $Aff(\mathbb{R})=\mathbb{R}^*\times \mathbb{R}$  is given by $(a,b)(c,d)=(ac,ad+b)$ with corresponding Lie algebra $\mathfrak{g}=span\{e_1^+,e_2^+\}$ where $e^+_{1,(x,y)}=x\frac{\partial}{\partial x}$ and $ e^+_{2,(x,y)}=x\frac{\partial}{\partial y}$ are left invariant  vector fields of $G$ with $[e_1^+,e_2^+]=e_2^+$.

Finding FLIAS on $G$ is equivalent to finding bilinear products on $\mathfrak{g}$ satisfying \eqref{E:zerocurvature1} and \eqref{E:freeoftorsion1}. Bilinear products on $\mathfrak{g}$ satisfying  \eqref{E:freeoftorsion1} can be written as
\[
\begin{array}{lclclc} e_1\cdot e_1=\alpha e_1+\beta e_2,&\qquad & e_1\cdot e_2=\gamma e_1+(\delta+1)e_2, \\ e_2\cdot e_1=\gamma e_1+\delta e_2 &\qquad& e_2\cdot e_2=\epsilon e_1+\lambda e_2\end{array}
\]
where $\alpha,\beta,\gamma,\delta,\epsilon$ and $\lambda$ are real numbers.

Such products verifying  \eqref{E:zerocurvature1} produce the system of polynomial equations  
\begin{align} \notag
&\beta\epsilon-\gamma\delta+\gamma=0\\ \notag
&\alpha\delta-\beta\gamma+\beta\lambda-\delta^2=0\\ \notag
&\alpha\epsilon-\gamma^2+\gamma\lambda-\delta\epsilon-2\epsilon=0\\ \notag
&\beta\epsilon-\gamma\delta-\lambda=0
\end{align}
These equations determine an algebraic variety $\mathcal{V}$.  Each point on $\mathcal{V}$ is a FLIAS on $G_0$, connected component of the unit of $G=Aff(\mathbb{R})$. The variety  $\mathcal{V}$ is a closed algebraic submanifold of $Hom(\mathfrak{g}\otimes\mathfrak{g},\mathfrak{g})$. By direct examination of the above equations one finds that $\mathcal{V}$  is the union of three irreducible components of dimension 2 which pairwise intersect along three non intersecting affine lines. These components are described as follows.

\noindent Component 1.  $\gamma=\delta=\epsilon=\lambda=0$ and we get the family of left symmetric products on $\mathfrak{g}$
\[
\begin{array}{lclclc}
 e_1\cdot e_1=\alpha e_1+\beta e_2,&\qquad& e_1\cdot e_2=e_2, \\
 e_2\cdot e_1=0, &\qquad& e_2\cdot e_2=0.\end{array}
\]

We will refer to this family as $\mathcal{F}_I(\alpha,\beta)$.

\noindent Component 2. $\gamma=\epsilon=\lambda=0$  and $\delta=\alpha$. In this case we obtain the family of left symmetric products
\[
\begin{array}{lclclc}
 e_1\cdot e_1=\alpha e_1+\beta e_2,&\qquad& e_1\cdot e_2=(\alpha+1)e_2, \\
 e_2\cdot e_1=\alpha e_2,&\qquad& e_2\cdot e_2=0.\end{array}
\]

We will label  this family as $\mathcal{F}_{II}(\alpha,\beta)$ with $\alpha\ne0$ to avoid repetition with $\mathcal{F}_I(0,\beta)$.
 
\noindent Component 3.  $\epsilon\ne0$,  $\alpha=2+\frac{\gamma^2}{\epsilon}$, $\beta=\frac{-\gamma^3-\gamma\epsilon}{\epsilon^2}$, $\delta=\frac{-\gamma^2}{\epsilon}$ and $\lambda=-\gamma$. Consequently we get
\[
\begin{array}{lclclc}
 e_1\cdot e_1=\left(2+\frac{\gamma^2}{\epsilon}\right) e_1-\left(\frac{\gamma^3+\gamma\epsilon}{\epsilon^2}\right) e_2,&\qquad& e_1\cdot e_2=\gamma e_1+\left(1-\frac{\gamma^2}{\epsilon}\right)e_2, \\
 e_2\cdot e_1=\gamma e_1-\frac{\gamma^2}{\epsilon} e_2,&\qquad& e_2\cdot e_2=\epsilon e_1-\gamma e_2\end{array}
\]
This family will be named as $\mathcal{F}_{III}(\gamma,\epsilon)$ with $\epsilon\ne0$.

 The planes below are idealized pictures of  the first two components. 

The  first plane  represents the Component 1 of $\mathcal{V}$ given above. Each vertical line provides one isomorphism class of FLIAS on $G_0$ except for $\alpha=1$. In this case there are two classes one of them is represented by the point $(1,0)$. 

Hence each point on the line $\beta=\alpha$ (on blue) is an isomorphism class. There is other class  at C:(1,0). In fact each point of any  line transversal to the foliation determined by vertical lines, for instance $\beta=c_1+c_2\alpha $ with $c_1$ and $c_2\ne0$ constants, is an isomorphism class.
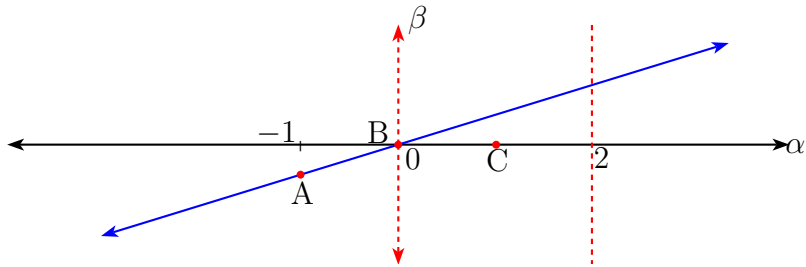
\begin{center}\psset{xunit=1.3cm,yunit=0.4cm,algebraic=true,dotstyle=o,dotsize=3pt 0,linewidth=0.8pt,arrowsize=3pt 2,arrowinset=0.25}
\begin{pspicture*}(-4.06,-3.95)(4.2,5)
\psline[linestyle=dashed,dash=2pt 2pt,linecolor=red]{->}(0,0)(0,4)
\psline{->}(0,0)(4,0)
\psline{->}(0,0)(-4,0)
\psline[linestyle=dashed,dash=2pt 2pt,linecolor=red]{->}(0,0)(0,-4)
\rput[tl](3.96,0.18){$\alpha$}
\rput[tl](0.1,4.62){$\beta$}
\rput[tl](0.9,-0.15){C}
\psline[linestyle=dashed,dash=2pt 2pt,linecolor=red](1.98,3.98)(1.98,-4.02)
\rput[tl](-1.1,-1.26){A}
\rput[tl](-0.32,0.76){B}
\rput[tl](0.07,-0.13){$0$}
\rput[tl](2,-0.1){$2$}
\rput[tl](-1.45,0.75){$-1$}
\psline[linecolor=blue]{->}(0,0)(3.38,3.38)
\psline[linecolor=blue]{->}(0,0)(-3.04,-3.04)
\begin{scriptsize}
\psdots[dotstyle=*,linecolor=red](1,0)
\psdots[dotstyle=|](-1,0)
\psdots[dotstyle=*,linecolor=red](0,0)
\psdots[dotstyle=*,linecolor=red](-1,-1)
\end{scriptsize}
\end{pspicture*}
\end{center}
The points on the dashed red lines ($\alpha=0$ and $\alpha=2$) are singular points of $\mathcal{V}$, the points  A:$(-1,-1)$, B:(0,0) and C:(1,0) represent special structures on $G_0$. The first one is a Lorentzian connection (that is, the Levi-Civita connection of a Lorentzian metric, see Subsection \ref{Ss:Lorentzian}) which is also symplectic (see Subsection \ref{S:simplecticstructures}), the point B corresponds to a complete connection (see Subsection \ref{S:geodesicallycomplete}) and the point C is a bi-invariant structure (in fact the unique bi-invariant connection within this component) in this case  the dimension of the group of affine transformations is at least 3 (see Theorem \ref{T:biinvariantstructureonlineargroup}, \cite{BM1}, \cite{M} and \cite{BM}). The points of this irreducible component, seen as groups of (classic) affine transformations of the plane, have a one-parameter subgroup of traslations ($N(\mathfrak{g})=\{x\in\mathfrak{g}\mid L_x=0\}$ is a line, where $L_x\in End(\mathfrak{g})$ is defined by $L_x(y)=xy$). Each left symmetric algebra within this family (with the exception of (1,1)) is isomorphic to the semidirect product of Lie algebras of the line $\mathbb{R}e_1$ with product $e_1\cdot e_1=\alpha e_1$, $\alpha\in\mathbb{R}$, by the bilateral ideal $N(\mathfrak{g})$ with the action given by the identity $e_1\cdot e_2=e_2$.  

The points on the second plane correspond to the irreducible Component 2 of $\mathcal{V}$ described above. As before, each point on the blue line gives one isomorphism class of affine structures on $G_0$. There is an extra class at the point $(-1,0)$. 
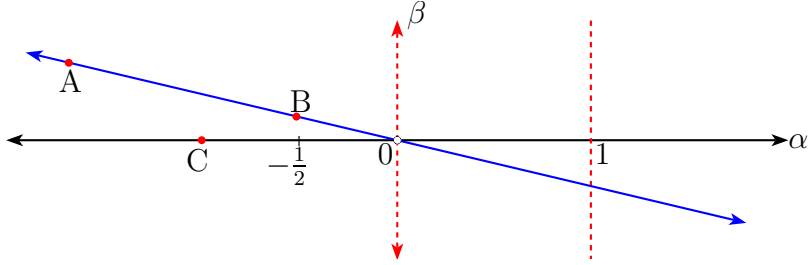
\begin{center}
\psset{xunit=1.3cm,yunit=0.4cm,algebraic=true,dotstyle=o,dotsize=3pt 0,linewidth=0.8pt,arrowsize=3pt 2,arrowinset=0.25}
\begin{pspicture*}(-4.12,-3.95)(4.24,4.6)
\psline[linestyle=dashed,dash=2pt 2pt,linecolor=red]{->}(0,0)(0,4)
\psline{->}(0,0)(4,0)
\psline{->}(0,0)(-4,0)
\psline[linestyle=dashed,dash=2pt 2pt,linecolor=red]{->}(0,0)(0,-4)
\rput[tl](4,0.18){$\alpha$}
\rput[tl](0.1,4.62){$\beta$}
\psline[linestyle=dashed,dash=2pt 2pt,linecolor=red](1.98,3.98)(1.98,-4.02)
\rput[tl](-3.46,2.34){A}
\rput[tl](-2.16,-0.3){C}
\rput[tl](-0.2,-0.12){$0$}
\rput[tl](2.02,-0.12){$1$}
\rput[tl](-1.34,-0.25){$-\frac{1}{2}$}
\rput[tl](-1.1,1.65){B}
\psline[linecolor=blue]{->}(0,0)(-3.8,2.92)
\psline[linecolor=blue]{->}(0,0)(3.57,-2.75)
\begin{scriptsize}
\psdots[dotstyle=|](-1,0)
\psdots(0,0)
\psdots[dotstyle=*,linecolor=red](-2,0)
\psdots[dotstyle=*,linecolor=red](-3.36,2.58)
\psdots[dotstyle=*,linecolor=red](-1.03,0.79)
\end{scriptsize}
\end{pspicture*}\end{center}
As above the dashed red lines ($\alpha=0$ and $\alpha=1$) are singular points of $\mathcal{V}$, the line $\alpha=0$ is the intersecting line with the first component. The points A:$(-2,2)$, B:$(-1/2,1/2)$  and C:$(-1,0)$ correspond to special structures on $G_0$. The first one is a Hessian connection (see Subsection \ref{S:hessian}), the point B is a symplectic connection relative to the symplectic structure determined by an open orbit of the co-adjoint representation of $G_0$ (see Section \ref{S:specialleftsymmetricproducts}),  and the point C is the affine bi-invariant structure determined by Equation \eqref{Eq:leftsymmetricproductbyasymplecticform} (the unique bi-invariant connection within this component) and the dimension of its group of affine transformations is at least 3. The structures on this component, seen as groups of affine transformations of the plane,  contain no traslations.

The irreducible Component 3 of $\mathcal{V}$ contains just two isomorphism classes of affine structures, given respectively by 
\[ \begin{array}{llcc} e_1\cdot e_1=2e_1 & e_1\cdot e_2=e_2\\ e_2\cdot e_1=0 & e_2\cdot e_2=\pm e_1 \end{array}\]
 according to whether $\epsilon>0$ or $\epsilon<0$. These structures, seen as affine transformations of the plane, contain no translations and they determine, by complexification, the same affine structure over $Aff(\mathbb{C})$. Both of these structures are also Hessian structures (see Subsection \ref{S:hessian}).

The two families and the four extra products displayed below describe all (real) isomorphism classes of FLIAS on $G_0$. The products $\mathcal{A}_1$ and $\mathcal{A}_2$ are asociative products and the complexifications $\mathcal{R}_1$ and $\mathcal{R}_2$ are isomorphic. Notice that these  products ($\mathcal{R}_1$ and $\mathcal{R}_2$) are real non-isomorphic products but the corresponding affine \'etale representations are  isomorphic (see Subsection \ref{S:affineetalerepresentations}).

\begin{align} \label{Eq:productosimetricofamilia1}
 \mathcal{F}_{1}(\alpha)=\mathcal{F}_{I}(\alpha,\alpha):\ &  \begin{cases}
    e_1\cdot e_1=\alpha e_1+\alpha e_2,\quad e_1\cdot e_2=e_2, & \hbox{} \\
    e_2\cdot e_1=0, \qquad\qquad\ \  e_2\cdot e_2=0, & \hbox{ }
  \end{cases}   \qquad\quad  \text{with }\alpha\in\mathbb{R}, \\  \label{Eq:productosimetricofamilia2}
 \mathcal{F}_{2}(\alpha)=\mathcal{F}_{II}(\alpha,-\alpha):\ &\begin{cases}
    e_1\cdot e_1=\alpha e_1-\alpha e_2,\ \ e_1\cdot e_2=(\alpha+1)e_2, & \hbox{} \\
    e_2\cdot e_1=\alpha e_2, \qquad\quad\ \  e_2\cdot e_2=0, & \hbox{ }
  \end{cases} \text{ with }\alpha\in\mathbb{R}\setminus\{0\}, \end{align}	
	\[ \mathcal{A}_{1}=\mathcal{F}_{I}(1,1):\quad\quad\  \begin{cases}
    e_1\cdot e_1=e_1,\quad\  e_1\cdot e_2=e_2, & \hbox{} \\
    e_2\cdot e_1=0, \quad\ \  e_2\cdot e_2=0, & \hbox{ }
  \end{cases} \]
  \[ \mathcal{A}_{2}=\mathcal{F}_{II}(-1,1):\quad  \begin{cases}
    e_1\cdot e_1=-e_1,\quad\ \  e_1\cdot e_2=0, & \hbox{} \\
    e_2\cdot e_1=-e_2, \quad\ \ \ e_2\cdot e_2=0, & \hbox{ }
  \end{cases} \]
	\[  \mathcal{R}_{1}=\mathcal{F}_{III}(0,1) :\   \begin{cases}
    e_1\cdot e_1=2e_1,\qquad e_1\cdot e_2=e_2, & \hbox{} \\
    e_2\cdot e_1=0, \qquad\quad e_2\cdot e_2=e_1, & \hbox{ }
  \end{cases} \]
	\[ \mathcal{R}_{2}=\mathcal{F}_{III}(0,-1):\   \begin{cases}
    e_1\cdot e_1=2e_1,\qquad e_1\cdot e_2=e_2, & \hbox{} \\
    e_2\cdot e_1=0, \qquad\quad e_2\cdot e_2=-e_1, & \hbox{ }
  \end{cases}
\]

\section{The  groups of affine transformations of the line as  symplectic  Lie groups} \label{S:lastsection}

We start the  section proving the following theorem. 

\begin{theorem} The Lie group $Aff(G,\nabla^+)$, with  $G=Aff(\mathbb{R})$ and  $\nabla^+$  any FLIAS  on $G$, is a flat  affine Lie group. 
\end{theorem} 
\begin{proof}
We use Theorem \ref{T:developant} to compute the group  $Aff(G,\nabla^+)$ case by case. We will make explicit a left symmetric product compatible with the bracket on the Lie algebra of $Aff(G,\nabla^+)$ for each isomorphism class of  $\nabla^+$.

From now on we will denote by $\nabla_{i}(\alpha)$ the connection determined by the product $\mathcal{F}_i(\alpha)$ with $i=1,2$ and by $\nabla_{i}$, $i=1,2,3,4$, the connections given by the products $\mathcal{A}_1$, $\mathcal{A}_2$, $\mathcal{R}_{1}$ and $\mathcal{R}_{2}$, respectively.  

\noindent 
Case 1. The Lie group $Aff(G,\nabla_1(\alpha))$, with $\alpha\notin\{0,1\}$, is given by 
\[ Aff(G,\nabla_1(\alpha))=\left\{ \phi:G\longrightarrow G\  \bigg|\begin{array}{c} \phi(x,y)=\left(a^{\frac{1}{\alpha}}x,bx^\alpha+\frac{\alpha}{\alpha-1}(a^{\frac{1}{\alpha}}-c)x+cy+d\right)\\ \quad\text{with}\quad a,c,b,d\in\mathbb{R},\ a>0\ \text{and}\ c\ne0\end{array}\right\}.\] 
It decomposes as the semidirect product $Aff(G,\nabla_1(\alpha))=\mathcal{N}\rtimes_\rho \mathcal{H}$ of the normal subgroup  $\mathcal{N}=\{\phi:G\longrightarrow G\mid \phi(x,y)=(a^{\frac{1}{\alpha}}x,\frac{\alpha}{\alpha-1}(a^{\frac{1}{\alpha}}-a)x+ay+b),\ a>0\}$ (isomorphic to $G_0$ as Lie groups) and the subgroup $\mathcal{H}=\{\phi:G\longrightarrow G\mid \phi(x,y)=(x,bx^\alpha+\frac{\alpha}{\alpha-1}(1-c)x+cy),\ c\ne0\}$ (isomorphic to $G$ as Lie groups) where $\mathcal{H}$ acts on $\mathcal{N}$ by conjugation.

Its Lie algebra $aff(G,\nabla_1(\alpha))=span\{e_1,e_2,e_3,e_4\}$ has nonzero brackets $[e_1,e_2]=-e_2$, $[e_2,e_3]=-e_2$, $[e_3,e_4]=e_4$ (the rest are zero or obtained by anti-symmetry). It follows that $aff(G,\nabla_1(\alpha))\cong \mathfrak{n}\rtimes_{ad}\mathfrak{h}$ where $ \mathfrak{n}=span\{e_1+e_3,e_4\}$, $\mathfrak{h}=span\{e_2,e_3\}$ are the Lie algebras of $\mathcal{N}$ and $\mathcal{H}$, respectively, and $\mathfrak{h}$ acts on  $\mathfrak{n}$ by the adjoint action. Then $\mathfrak{n}\rtimes_{ad}\mathfrak{h}$ is an LSA with the product given by 
\begin{equation}\label{Eq:semidirectproductoflsas} (n_1,h_1)\bullet(n_2,h_2)=(n_1\cdot n_2+ad_{h_1}(n_2),h_1*h_2)\end{equation}
where $\cdot$ is the left symmetric on $\mathfrak{n}$ given by  \eqref{Eq:productosimetricofamilia1} or \eqref{Eq:productosimetricofamilia2} and $*$ is any left symmetric product on $\mathfrak{h}$.

\noindent
Case 2. The group $Aff(G,\nabla_2(\alpha))$, with $\alpha\ne-1$,  is given by 
\[ Aff(G,\nabla_2(\alpha))=\left\{ \phi:G\longrightarrow G\ \bigg|\begin{array}{c} \phi(x,y)=\left(a^{\frac{1}{\alpha}}x,bx^{-\alpha}+\frac{\alpha}{\alpha+1}(a^\frac{1}{\alpha}-c)x+cy+d\right)\\ \quad\text{with}\quad a,c,b,d\in\mathbb{R},\ a>0\ \text{and}\ c\ne0\end{array}\right\}.\] 
It decomposes as the semidirect product $Aff(G,\nabla_2(\alpha))=\mathcal{N}\rtimes_\rho \mathcal{H}$ of the normal subgroup $\mathcal{N}=\{\phi:G\longrightarrow G\mid \phi(x,y)=(a^{\frac{1}{\alpha}}x,\frac{b}{\alpha}x^{-\alpha}+\frac{\alpha}{\alpha+1}(a^\frac{1}{\alpha}-1)x+y),\ a>0\}$ (isomorphic to $G_0$) and the subgroup $\mathcal{H}=\{\phi:G\longrightarrow G\mid \phi(x,y)=(x,\frac{\alpha}{\alpha+1}(1-c)x+cy+d),\ c\ne0\}$ (isomorphic to $G$) where $\mathcal{H}$ acts on $\mathcal{N}$ by conjugation.

The Lie algebra is given by $aff(G,\nabla_2(\alpha))=span\{e_1,e_2,e_3,e_4\}$ and has nonzero brackets $[e_1,e_2]=e_2$, $[e_2,e_3]=-e_2$, $[e_3,e_4]=e_4$. It follows that $aff(G,\nabla_2(\alpha))\cong \mathfrak{n}\rtimes_{ad}\mathfrak{h}$ where $ \mathfrak{n}=span\{e_1,e_2\}$, $\mathfrak{h}=span\{e_3,e_4\}$ are the Lie algebras of $\mathcal{N}$ and $\mathcal{H}$, respectively, and $\mathfrak{h}$ acts on  $\mathfrak{n}$ by the adjoint action. Then $\mathfrak{n}\rtimes_{ad}\mathfrak{h}$ is an LSA with the product given by \eqref{Eq:semidirectproductoflsas}, where the left symmetric product   on $\mathfrak{n}$ is given by  \eqref{Eq:productosimetricofamilia1} or \eqref{Eq:productosimetricofamilia2} and  any left symmetric product on $\mathfrak{h}$.

\noindent
Case 3. For $i=1,2$, the elements of  $Aff(G,\nabla_i((-1)^{i+1}))$  are diffeomorphisms of $G$ given by $\phi(x,y)=\left(ax,(c-a)x\ln x+bx+cy+d\right)$ where $a,c,b,d\in\mathbb{R},\ a>0$ and $c\ne0$. This group  decomposes as the semidirect product $Aff(G,\nabla_i((-1)^{i-1})=\mathcal{N}\rtimes_\rho\mathcal{H}$ of the normal subgroup  $\mathcal{N}=\{\phi\mid 
\phi(x,y)=(ax,(1-a)x\ln x+bx+y),\ a>0\}\cong G_0$ and the subgroup $\mathcal{H}=\{\phi\mid \phi(x,y)=(x,(c-1)x\ln x+cy+d),\ c\ne0\}\cong G$ where $\mathcal{H}$ acts on $\mathcal{N}$ by conjugation.

The Lie algebra  $aff(G,\nabla_i((-1)^{i+1}))=span\{e_1,e_2,e_3,e_4\}$ of $Aff(G,\nabla_i)$ has nonzero brackets $[e_1,e_2]=-e_2$, $[e_1,e_3]=-e_2$, $[e_2,e_3]=-e_2$, $[e_3,e_4]=e_4$. We also get that $aff(G,\nabla_i((-1)^{i+1}))\cong \mathfrak{n}\rtimes_{ad}\mathfrak{h}$ where $ \mathfrak{n}=span\{e_1,e_2\}$, $\mathfrak{h}=span\{e_3,e_4\}$ are the Lie algebras of $\mathcal{N}$ and $\mathcal{H}$, respectively, and $\mathfrak{h}$ acts on  $\mathfrak{n}$ by the adjoint action.  Then $\mathfrak{n}\rtimes_{ad}\mathfrak{h}$ is an LSA with the product given by \eqref{Eq:semidirectproductoflsas}, where the left symmetric product   on $\mathfrak{n}$ is given by  \eqref{Eq:productosimetricofamilia1} with  $\alpha=1$ or by \eqref{Eq:productosimetricofamilia2} with  $\alpha=-1$ and  any left symmetric product on $\mathfrak{h}$.

\noindent
Case 4. For $i=1,2$, the group  $Aff(G,\nabla_i)=\{\phi\mid \phi(x,y)=(ax,bx+cy+d),\ a>0,\ c\ne0\}.$  One can check that $Aff(G,\nabla_i)=\mathcal{N}\rtimes_\rho\mathcal{H}$ with  $\mathcal{N}=\{\phi\mid \phi(x,y)=(ax,bx+y),\ a>0\}$ (isomorphic to $G_0$),  $\mathcal{H}=\{\phi\mid \phi(x,y)=(x,cy+d),\ c\ne0\}$ (isomorphic to $G$) and the action of $\mathcal{H}$ on $\mathcal{N}$ given by conjugation. Its Lie algebra is given by $aff(G,\nabla_i)=span\{e_1,e_2,e_3,e_4\}$ with non-zero brackets given by $[e_1,e_2]=-e_2$, $[e_2,e_4]=-e_2$ and $[e_3,e_4]=e_4$ is the semidirect product $\mathfrak{n}\rtimes_{ad}\mathfrak{h}$ of the Lie subalgebras $\mathfrak{n}=span\{e_1,e_2\}$ and $\mathfrak{h}=span\{e_3,e_4\}$. These are, respectively, the Lie algebras of $\mathcal{N}$ and $\mathcal{H}$.

In all cases so far we get the exact sequence of Lie groups
\[ \begin{array}{ccccccccc} 
\epsilon&\longrightarrow & \mathcal{N}\cong G_0 &\overset{i}\longrightarrow &Aff(G,\nabla)&\overset{\pi}\longrightarrow &Aff(G,\nabla)/\mathcal{N}\cong \mathcal{H}\cong G&\longrightarrow &\epsilon 
\\  & &\sigma &\mapsto & \sigma & & & & \\  & & & & \phi=\sigma\circ\varphi & \mapsto &\varphi & & \end{array} \]

The sequence  splits since $\pi\circ i=id_{\mathcal{H}}$ where $i:\mathcal{H}\longrightarrow Aff(G,\nabla)$ is the inclusion map. In fact it is an exact sequence of flat affine Lie groups (compare with Example 3.6 of \cite{MR}).

\noindent
Case 5. By Remark \ref{R:geodesicallycomplete}, the connection $\nabla_1(0)$ is complete , hence $Aff(G,\nabla_1(0))$ is isomorphic as Lie group to $Aff(\mathbb{R}^2)$. Therefore it admits an affine (symplectic) structure (see \cite{BM}).

\noindent
Case 6. For $j=3,4$, the Lie group $Aff(G,\nabla_j)$ is given by $$Aff(G,\nabla_j)=\left\{ \phi:G\longrightarrow G\mid \phi(x,y)=\left(ax,ay+b\right)\text{ where }a,b\in\mathbb{R}\text{ and } a>0\right\}.$$
Hence it is a Lie group isomorphic to $G_0$ and therefore it is an affine group.
\end{proof}

\begin{remark} In Cases 1 through 4 of the previous proof it can be verified that $Aff(G,\nabla)$ is isomorphic, as a Lie group, to the direct product of Lie groups $Aff(\mathbb{R})\times Aff(\mathbb{R})$. Therefore in those cases we have that 
\[ \dim[Aff(Aff(\mathbb{R}),\nabla^+)]=4. \] 
In Case 5. we have that $\dim[Aff(Aff(\mathbb{R}),\nabla^+)]=6$. Finally, in Case 6. we get that $\dim[Aff(Aff(\mathbb{R}),\nabla^+)]=2.$  
\end{remark}


We conclude the section with the following result.

\begin{theorem} The group $Aff(G,\nabla^+)$, where $\nabla^+$ is  any FLIAS  on $G=Aff(\mathbb{R})$, is a symplectic Lie group. 
\end{theorem}
\begin{proof} By the previous proof $Aff(G,\nabla^+)$ has  dimension 2, 4 or 6. 

If $Aff(G,\nabla^+)$ is 2-dimensional, it follows that $Aff(G,\nabla^+)$ is isomorphic as Lie group to $Aff(\mathbb{R},\nabla^0)$ and hence it is symplectic (see Subsection \ref{S:simplecticstructures}).

If $Aff(G,\nabla^+)$ is 6-dimensional, we get that $Aff(G,\nabla^+)$ is isomorphic to $Aff(\mathbb{R}^2,\nabla^0)$, thus it is symplectic (see \cite{BM}). 

If $Aff(G,\nabla^+)$ is 4-dimensional, it follows from the proof of the previous theorem that $Aff(G,\nabla^+)$ is locally isomorphic to $Aff(\mathbb{R},\nabla^0)\rtimes_{Ad}Aff(\mathbb{R},\nabla^0),$ where $Ad$ is the adjoint action. That is, $Aff(G,\nabla^+)$ is the double Lie group of $Aff(\mathbb{R},\nabla^0)$ and therefore it is symplectic (see \cite{DM2}). 
These  groups are also pseudo-K\"ahler Lie groups, see Theorem \ref{P:cotangentaskalerliegroup}.
\end{proof}

\section{Special flat left invariant affine structures on $Aff(\mathbb{R})$} \label{S:specialleftsymmetricproducts}
Among the FLIAS on $Aff(\mathbb{R})$ we identify those   of special interest in geometry. More specifically we identify affine pseudo-Riemannian structures,  affine symplectic structures,  affine K\"ahler structures, and affine Hessian structures.

\subsection{Complete Structures } \label{S:geodesicallycomplete}

It is known that a FLIAS on a Lie group $G$  determined by a left symmetric product  in $\mathfrak{g}=Lie(G)$, is geodesically complete if and only if $tr(R_b)=0$ for all $b\in\mathfrak{g}$, where $R_b$ is the matrix of the transformation defined by $R_b(a)=a\cdot b$ (see \cite{H} and \cite{M}). A direct calculation shows that $\nabla_1(0)$ is the unique structure on $Aff(\mathbb{R})$ with this property (anyone can verify this using the equations of the geodesics as well).

\begin{remark} \label{R:geodesicallycomplete} Let $G$ be the Lie group of affine transformations of the line. There exists a unique geodesically complete FLIAS $\nabla$ on $G$. This connection is given by
\[ \nabla_{e_1^+}e_1^+=0,\qquad\nabla_{e_1^+}e_2^+=e_2^+,\qquad\nabla_{e_2^+}e_1^+=0\quad\text{and}\quad
\nabla_{e_2^+}e_2^+=0.\]
Geodesics for this connection are given below in Equations \eqref{Eq:geodesicacompleta} and \eqref{Eq:geodesicacompletaa=0}.
\end{remark}

\subsection{Hessian structures} \label{S:hessian} Recall that an affine manifold is pseudo-Hessian if there exists a pseudo-Riemannian metric which in local affine coordinates is the Hessian of a function (see \cite{Kz} and \cite{Sh}). For an affine Lie group $(G,\nabla^+)$ this means that there is a left invariant pseudo-Riemannian metric  on ${G}$ verifying the identity
  $$\langle \nabla_{X^+}^+Y^+,Z^+\rangle-\langle X^+,\nabla_{Y^+}^+Z^+\rangle=\langle \nabla_{Y^+}^+X^+,Z^+\rangle-\langle Y^+,\nabla_{X^+}^+Z^+\rangle$$
The affine structures on $Aff(\mathbb{R})$ determined by  $\nabla_1(-1),\ \nabla_2(-2),\ \nabla_3 $ and $\nabla_4$ are Hessian  relative to the metrics $\langle\ ,\ \rangle$  given by the matrices $A_i=M\left(\langle\ ,\ \rangle,\{e_1,e_2\}\right)$, respectively, where $A_1=\left[\begin{matrix}0&1\\1&0\end{matrix}\right]$, $A_2=\left[\begin{matrix}1&-1/4\\-1/4&\ \ 1/8\end{matrix}\right]$, $A_3=\left[\begin{matrix}2&0\\0&1\end{matrix}\right]$ and $A_4=\left[\begin{matrix}2&\ \ 0\\0&-1\end{matrix}\right]$. The structures $\nabla_2(-2)$ and $\nabla_3 $  are both  Riemannian while  the structures $\nabla_1(-1)$ and $\nabla_4$ are pseudo-Riemannian.

\begin{lemma} The symplectic cotangent Lie group $T^*G$ of $(G,\nabla^+,g^+)$, where $\nabla^+$ is one of $\nabla_1(-1),\ \nabla_2(-2),\ \nabla_3 $ and $\nabla_4$ and $g^+$ the corresponding metric given above, is a pseudo-K\"ahler Lie group.
\end{lemma}
\begin{proof} The proof follows from Theorem \ref{P:cotangentaskalerliegroup}.
\end{proof}

\subsection{Lorentzian structures} \label{Ss:Lorentzian} Let $\langle\ ,\ \rangle$ be the left invariant Lorentzian metric on $Aff(\mathbb{R})$ given by the  matrix  $M(\langle\ ,\ \rangle,\{e_1,e_2\})=\left[\begin{matrix}\ \ 1&-1\\-1&\ \ 0\end{matrix}\right]$.  The corresponding Levi-Civita conexion    $\nabla_L=\nabla_1(-1)$ is flat.

 The group $Aff(G,\nabla_1(-1))$, where $G=Aff(\mathbb{R})$, is a 4-dimensional flat affine Lie group (see Section \ref{S:lastsection}). The group of diffeomorphisms of $Aff(\mathbb{R})$ preserving both the affine and the Lorentzian structures of $(Aff(\mathbb{R}),\nabla_L)$ is the 2-dimensional Lie group isomorphic to $Aff(\mathbb{R})$ given by
\[ \{\phi:Aff(\mathbb{R})\longrightarrow Aff(\mathbb{R})\mid \phi(x,y)=(ax,ay+b) \text{ with } a,b\in\mathbb{R}\text{ and }a\ne0\}.\]

\subsection{Flat left invariant affine symplectic  structures} The co-adjoint action of the group $Aff(\mathbb{R})$  on $aff(\mathbb{R})^*$, the dual space of  $aff(\mathbb{R})$, is given by $$\sigma\cdot(\alpha e_1^*+\beta e_2^*)=Ad^*(\sigma)(\alpha e_1^*+\beta e_2^*)=\left(\alpha+\frac{b\beta}{a}\right)e_1^*+\frac{\beta}{a}e_2^*,$$ where $\sigma=(a,b)$. This action has two open orbits and a line of fixed points. Since every open orbit of the co-adjoint representation is a symplectic manifold, the pullback by the orbital map on the $Orb(\epsilon)$  defines a left invariant symplectic form on $Aff(\mathbb{R})$  given by  $w^+(e_1^+,e_2^+)=1$. The formula 
\begin{equation}\label{Eq:leftsymmetricproductbyasymplecticform} \omega^+((a\cdot b)^+,c^+)= -\omega^+(b^+,[a^+,c^+]).\end{equation}
 determines the  affine structure  $\nabla_2$. Notice that $(Aff(\mathbb{R})_0,\nabla_2)$ is a symplectic Lie group in the sense of Lichnerowicz-Medina (see \cite{LM}).

\subsection{Symplectic Connections} \label{S:simplecticstructures} Recall that a symplectic connection  on a symplectic manifold $(M,\Omega)$ is a linear connection $\nabla$ satisfying the condition 
\[ \Omega\left(\nabla_XY,Z\right)+\Omega\left(Y,\nabla_XZ\right)=X\cdot \Omega(Y,Z)\quad\text{for all}\quad X,Y,Z\in\mathfrak{X}(M).\] 
Let $(G,\omega^+)$ be a symplectic Lie group. That $\nabla$ is a symplectic flat left invariant  linear connection,  means that there exists a bilinear product on $\mathfrak{g}$ verifying $\omega(x y,z)+\omega(y,x z)=0$ for all $x,y,z\in\mathfrak{g}$. 
A direct calculation shows that   $\nabla_1(-1)$ and $\nabla_2(-1/2)$ on $(Aff(\mathbb{R}),\omega^+)$ are symplectic flat affine connections. 

It can also be verified that $(Aff(G),g^+,j^+)$ is a K\"ahler   Lie group by taking the left invariant Riemannian metric $g^+$ with matrix  $M(\langle\ ,\ \rangle,\{e_1,e_2\})=\left[\begin{matrix}1&0\\0&1\end{matrix}\right]$ and the left invariant complex structure determined by $j^+(e_1)=-e_2$ and $j^+(e_2)=e_1$. This structure  is holomorphic  non-Ricci Flat    (see \cite{LM} and \cite{L}).


\section{appendix} \label{S:apendix}

\subsection{\'Etale Affine Representations of $Aff(\mathbb{R})_0$} \label{S:affineetalerepresentations} Let ${G}$ be a connected  Lie group and $\mathfrak{g}=Lie(G)$ its Lie algebra. A bilinear product $L_a(b)=a b$ on $\mathfrak{g}$ is a left symmetric product compatible with the bracket in  $\mathfrak{g}$ if and only if the map $\theta: \mathfrak{g}\longrightarrow \mathfrak{g}$ defined by $ \theta( x)=(x,L_x)$ is a representation of the Lie algebra $\mathfrak{g}$ by (classical) affine endomorphisms of $\mathfrak{g}$.   The exponential map determines a unique Lie group homomorphism  $\rho:\widehat{G}\longrightarrow Aff(\mathfrak{g})=\mathfrak{g}\rtimes GL(\mathfrak{g})$. If $\rho$ is given by $\rho(\sigma)=(Q(\sigma),F_\sigma),$ then $\sigma\rightarrow Q(\sigma)$ is a 1-cocycle of $\widehat{G}$ relative to the linear representation $\sigma\mapsto F_\sigma$.  The orbital map $\pi:\widehat{G}\longrightarrow Orb(0)$ defined by $\pi(\sigma)=\rho(\sigma)(0)$ is a local  diffeomorphism equivariant by the actions of $\widehat{G}$ on itself by left multiplications and of the group $\rho(\widehat{G})$ on $\mathfrak{g}$ by affine transformations. So the $\pi$-pullback $\nabla^+$ of $\nabla^0$ is a FLIAS. Moreover, this structure on $\widehat{G}=G$ is geodesically complete if and only if the action associated to $\rho$ is transitive (see \cite{M}).

It is easy to see that   $Q(\sigma):G\longrightarrow\mathfrak{g}$ is a developing map of $(G,\nabla^+)$.

It is easy to check that the exponential map of $\mathfrak{g}=aff(\mathbb{R})$ is given by
\[ \begin{array}{ccclc}exp:&\mathfrak{g}&\longrightarrow & \widehat{G}\\&(a,b)&\mapsto& \begin{cases}\left(e^{a},\frac{b}{a}\left(e^{a}-1\right)\right),&\text{for } a\ne0\\ 
(1,b)&\text{for } a=0\end{cases}\end{array}.\] 
Using this and Theorem \ref{T:developant} (the Development Theorem), one can find that the corresponding affine \'etale representations of $G_0=Aff(\mathbb{R})_0$, for each  FLIAS, are as follows

\noindent \textbf{Family $\mathcal{F}_1(\alpha)$.}
For  $\alpha\notin\{0,1\}$
\[ \begin{array}{clclclc} \rho_1(\alpha):&G_0&\longrightarrow &Aff(\mathfrak{g})\subseteq GL(\mathfrak{g}\oplus\mathbb{R})\\& (x,y)&\mapsto&
\left[\begin{matrix}x^{\alpha}&0&\frac{1}{\alpha}(x^{\alpha}-1)\\ \frac{\alpha}{\alpha-1}(x^\alpha-x)&x&y+\frac{1}{\alpha-1}(x^\alpha-\alpha x)+1\\0&0&1\end{matrix}\right]
\end{array} \] 
when $\alpha=1$
\[ \begin{array}{clclclc} \rho_1:&G_0&\longrightarrow &Aff(\mathfrak{g})\\& (x,y)&\mapsto& 
\left[\begin{matrix}x&0&x-1\\x\ln x&x&1+y+x(\ln x-1)\\0&0&1\end{matrix}\right]\end{array} \]
when $\alpha=0$ we get the  representation
\[ \begin{array}{clclclc} \rho_1(0):&G_0&\longrightarrow &Aff(\mathfrak{g})\\& (x,y)&\mapsto&
\left[\begin{matrix}1&0&\ln x\\0&x&y\\0&0&1\end{matrix}\right]
\end{array} \]

\noindent \textbf{Family $\mathcal{F}_2(\alpha)$.} For $\alpha\notin\{0,-1\}$ 
\[ \begin{array}{clclclc} \rho_2(\alpha):&G_0&\longrightarrow &Aff(\mathfrak{g})\\& (x,y)&\mapsto&
\left[\begin{matrix}x^{\alpha }&0&\frac{1}{\alpha}(x^{\alpha }-1)\\ \alpha x^{\alpha }(y -x+1)&x^{\alpha +1}&x^{\alpha}(y+1)-\frac{1}{\alpha+1}(\alpha x^{\alpha+1}+1) \\0&0&1\end{matrix}\right]
\end{array} \]

and when $\alpha=-1$, we get the  representation
\[ \begin{array}{clclclc} \rho_2(-1):&G_0&\longrightarrow &Aff(\mathfrak{g})\\&  (x,y)&\mapsto&
\left[\begin{matrix}1/x&0&1-1/x\\  1-y/x-1/x&1&1/x+y/x-1+\ln x\\ 0&0&1\end{matrix}\right]
\end{array} \]

\noindent \textbf{ $\mathcal{A}_1$}
\[ \begin{array}{clclclc} \rho_1:&G_0&\longrightarrow &Aff(\mathfrak{g})\\& (x,y)&\mapsto& 
\left[\begin{matrix}x&0&x-1\\0&x&y\\0&0&1\end{matrix}\right]\end{array}
\]

\noindent \textbf{ $\mathcal{A}_2$} 
\[ \begin{array}{clclclc} \rho_2:&G_0&\longrightarrow &Aff(\mathfrak{g})\\& (x,y)&\mapsto&
\left[\begin{matrix}1/x&0&1-1/x\\ -y/x&1&y/x\\0&0&1 \end{matrix}\right] \end{array}
\]

\noindent \textbf{ $\mathcal{R}_1$}
\[ \begin{array}{clclclc} \rho_3:&G_0&\longrightarrow &Aff(\mathfrak{g})\\& (x,y)&\mapsto&
\left[\begin{matrix}x^{2}&xy&\frac{1}{2}(x^2+y^2-1)
\\0&x&y\\0&0&1\end{matrix}\right]\end{array} \]

\noindent \textbf{ $\mathcal{R}_2$}
\[ \begin{array}{clclclc} \rho_4:&G_0&\longrightarrow &Aff(\mathfrak{g})\\& (x,y)&\mapsto&
\left[\begin{matrix}x^{2}&-xy&\frac{1}{2}(x^2-y^2-1)
\\0&x&y\\0&0&1\end{matrix} \right]\end{array} \]

Notice that the $(\mathfrak{g},\mathcal{R}_1)$ and $(\mathfrak{g},\mathcal{R}_2)$ are not isomorphic real LSA's. However their corresponding affine \'etale representations $\rho_3$ and $\rho_4$ are  isomorphic.

\subsection{Geodesics in $Aff(\mathbb{R})_0$} \label{S:geodesics}

 In terms of global coordinates $x$ and $y$, geodesics  $\gamma(t)=(x(t),y(t))$ are  solutions to the system of differential equations
\[
\begin{cases}
x''(t)+\Gamma_{1,1}^1 [x'(t)]^2+(\Gamma_{1,2}^1+\Gamma_{2,1}^1 )x'(t)y'(t)+\Gamma_{2,2}^1 \left[y'(t)\right]^2=0 & \\
y''(t)+\Gamma_{1,1}^2 (x'(t))^2+(\Gamma_{1,2}^2+\Gamma_{2,1}^2 )x'(t)y'(t)+\Gamma_{2,2}^2 (y'(t))^2=0
\end{cases} \] 
where $\Gamma_{ij}^k$ denote the Christoffel symbols of the connection   
\[ (e_i\cdot e_j)^+=\displaystyle{\nabla_{e_i ^+}^+ e_j ^+=\sum_{k=1}^2\Gamma_{ij}^ke_k^+.} \] 
Since  we are considering only left invariant connections, it is enough to know geodesics through the group identity $\epsilon=(1,0)$, i.e., verifying the  initial conditions $(x(0),y(0))=(1,0)$ and $(x'(0),y'(0))=(a,b)$.

Below we display, for each affine structure in $Aff(\mathbb{R})_0$, the geodesics through $\epsilon$. 

\noindent \textbf{Geodesics for} $\bf{\nabla_{1}(\alpha)}.$ For $a\ne0$

\begin{align} \notag
&\begin{cases} x(t)=1+\frac{1}{\alpha}\ln\vert a\alpha t+1\vert & \\ 
 y(t)=\frac{b(\alpha-1)-\alpha a}{a(\alpha-1)^2}\left[(a\alpha t+1)^{\frac{\alpha-1}{\alpha}}-1\right]+\frac{1}{\alpha-1}\ln|a\alpha t+1| & \end{cases}  \text{if } \alpha\notin\{0,1\}. \\ \notag
& \begin{cases} x(t)=1+\ln\vert at+1\vert & \\ y(t)=\frac{1}{2a}\ln\vert at+1\vert\left[2b-a\ln\vert at+1\vert\right] & \end{cases}\qquad\qquad\qquad\quad  \text{if } \alpha=1 \\ \label{Eq:geodesicacompleta} &\begin{cases} x(t)=at+1 &\\ y(t)=\frac{b}{a}\left(1-e^{-at}\right), & \end{cases} \qquad\qquad\qquad\qquad\qquad\qquad\qquad\  \text{ if } \alpha=0. \end{align}

  and for $a=0$ we get
 \begin{equation} \label{Eq:geodesicacompletaa=0} \begin{array}{cccc}  x(t)=1&\text{ and }&y(t)=bt& \text{for any } \alpha \end{array}\end{equation} 

\noindent  \textbf{Geodesics for} $\bf{\nabla_2(\alpha)}.$ For $a\ne0$
\[ \begin{array}{lcc} \begin{cases} x(t)=1+\frac{1}{\alpha}\ln\vert a\alpha t+1\vert & \\
y(t)=\frac{b(\alpha+1)+\alpha a}{(\alpha+1)^2a}\left[1-(\alpha at+1)^{-\frac{2\alpha+1}{\alpha}}\right]-\frac{1}{\alpha+1}\ln|\alpha at+1| & \end{cases} &\text{ if } \alpha\ne-1
 \\ \ & \ \\ \begin{cases}  x(t)=1-\ln\vert at-1\vert & \\ y(t)=\frac{-1}{2a}\ln\vert at-1\vert\left[a\ln\vert at-1\vert+2b\right] & \end{cases} & \text{ if } \alpha=-1
\end{array} \] 

and for $a=0$ we get
\[ x(t)=1\qquad\text{and}\qquad y(t)=bt,\qquad \text{for any }\alpha.\] 

\noindent  \textbf{Geodesics for } $\mathbf{\nabla_1}.$ For $a\ne0$ 
\[ x(t)=1+\ln\vert at+1\vert \quad\text{and}\quad y(t)=\frac{b}{a}\ln\vert at+1\vert.\]
and for $a=0$ we get
\[ x(t)=1\qquad\text{and}\qquad y(t)=bt.\] 

\noindent  \textbf{Geodesics for } $\mathbf{\nabla_2}.$ For   $a\ne0$
\[  x(t)=1-\ln\vert 1-at\vert \qquad \text{and} \qquad y(t)=\frac{-b}{a}\ln\vert1-at\vert  \]

and for   $a=0$ we get
\[ x(t)=1\qquad\text{and}\qquad y(t)=bt.\] 

\noindent  \textbf{Geodesics for } $\mathbf{\nabla_3}.$ For $a$ and $b$ no both zero
\begin{align} \notag x(t)&=\frac{1}{2}ln|-b^2t^2+2at+1|+1 \\\notag  y(t)&=arccos\left(\frac{\sqrt{a^2+b^2-(b^2t-a)^2}}{\sqrt{a^2+b^2}}\right)-arccos\left(\frac{|b|}{\sqrt{a^2+b^2}}\right).\end{align}
If $a=b=0$ we get $x(t)=1$ and $y(t)=0$.

\noindent  \textbf{Geodesics for } $\mathbf{\nabla_4}.$ For $a\geq|b|$
\begin{align} \notag x=&\dfrac{1}{2}\ln\vert b^2t^2+2at+1\vert+1\\\notag
 y=&\ln\left\vert b^2t+a+\vert b\vert(|b^2t^2+2at+1|)^{1/2}\right\vert-\ln\vert a+\vert b\vert\vert
\end{align}
 and for $|b|> a$ we have 
\begin{align} \notag x=&\dfrac{1}{2}\ln\vert b^2t^2+2at+1\vert+1\\\notag
 y=&\frac{1}{2}\ln\left|\frac{|b|\sqrt{|b^2t^2+2at+1|}+b^2t+a|}{|b|\sqrt{|b^2t^2+2at+1|}-b^2t-a}\right|-\frac{1}{2}\ln\left|\frac{|b|+a}{|b|-a}\right|
\end{align}

\noindent
Acknowledgment

 We are grateful to the referee of this paper for his suggestions and comments which helped to improve the presentation of this work.

\end{document}